\documentclass[12pt,a4paper]{article}

\usepackage[utf8]{inputenc}
\usepackage[english]{babel}

\usepackage{amsmath}
\usepackage{color}
\usepackage{amsfonts,enumitem}
\usepackage{amssymb}
\usepackage{hyperref} 
\usepackage{graphicx,tikz}
\usepackage{cleveref}
\usepackage{bm}
\usepackage{bbm}

\usepackage[margin=1in,letterpaper]{geometry} 

\newcommand*{\SET}[1]  {\ensuremath{\mathbb{#1}}}

\newcommand{\R}{\SET{R}}

\newcommand{\E}{\SET{E}}

\newcommand{\N}{\SET{N}}

\newcommand{\conv}{\operatorname{conv}}

\newcommand{\Jac}{\operatorname{Jac}}

\DeclareMathOperator{\di}{d\!}
\newcommand{\Ran}{\R_{\operatorname{an}}}
\newcommand{\Ranexp}{\R_{\operatorname{an,exp}}}

\newcommand{\diff}{\operatorname{diff}}

\newtheorem{theorem}{Theorem}[section]

\newtheorem{lemma}{Lemma}[section]
\newtheorem{proposition}{Proposition}[section]
\newtheorem{corollary}{Corollary}[section]

\newtheorem{definition}{Definition}[section]
\newtheorem{remark}{Remark}[section]

\newtheorem{assumption}{Assumption}[section]

\newenvironment{proof}[1][]{\noindent {\bf Proof#1:\;}}{\hfill $\Box$}

\providecommand{\keywords}[1]{\textbf{\textbf{Keywords. }} #1}

\providecommand{\ams}[1]{\textbf{\textbf{AMS subject classifications. }} #1}

\begin{document}

\title{Nonsmooth nonconvex stochastic heavy ball}

\author{Tam Le \thanks{Toulouse School of Economics, Universit\'e de Toulouse, ANITI.}}
\maketitle


\maketitle

\begin{abstract}
    Motivated by the conspicuous use of momentum-based algorithms in deep learning, we study a nonsmooth nonconvex stochastic heavy ball method and show its convergence. Our approach builds upon semialgebraic (definable) assumptions commonly met in practical situations and combines a nonsmooth calculus with a differential inclusion method. Additionally, we provide general conditions for the sample distribution to ensure the convergence of the objective function. Our results are general enough to justify the use of subgradient sampling in modern implementations that heuristically apply rules of differential calculus on nonsmooth functions, such as backpropagation or implicit differentiation. As for the stochastic subgradient method, our analysis highlights that subgradient sampling can make the stochastic heavy ball method converge to artificial critical points. Thanks to the semialgebraic setting, we address this concern showing that these artifacts are almost surely avoided when initializations are randomized, leading the method to converge to Clarke critical points.
\end{abstract}


\ams{49J53, 68Q25, 68W27, 49J52, 28B20}


\section{Introduction}
In order to deal with a stochastic minimization problem, we are interested in a variant of the gradient method which includes a momentum term, called the heavy ball method \cite{POLYAK19641}. In the case of a stochastic minimization problem involving an intractable expectation,  this method may be refined using gradient sampling instead of a deterministic gradient. This stochastic version of the method and more generally momentum-based algorithms with gradient sampling are widely used in deep learning due to their empirical effectiveness \cite{krizhevsky,pmlr-v28-sutskever13,Kingma2014AdamAM}. Gradient sampling is usually justified by the interchanging of the differentiation and the expectation operators, i.e., the expectation of the gradient is the gradient of the expectation. In this case, many convergence results exist for the stochastic heavy ball method, in the convex setting  \cite{Yang2016UnifiedCA,pmlr-v115-orvieto20a,pmlr-v134-sebbouh21a}, and in the nonconvex setting \cite{gadat2018stochastic} via the ODE approach \cite{benaim1999dynamics}.

In many situations in deep learning, the objective function is nonconvex and nonsmooth. In this context, the differential inclusion approaches developed in \cite{benaim,bianchi2021closed} have been used to study first-order methods in the nonsmooth and nonconvex setting \cite{Davis2020,Majewski2018AnalysisON,bolte2022subgradient,castera2021inertial,Bolte2020LongTD}. Since they rely on a set-valued ODE driven by a convex-valued map having a closed graph, the Clarke subdifferential \cite{clarke1990optimization} naturally arises in the choice of a nonsmooth first-order oracle. For a locally Lipschitz integrand, one may define the Clarke subdifferential \cite{clarke1990optimization} and consider a nonsmooth version of the stochastic heavy ball method.

In order to justify the convergence of such a nonsmooth version to Clarke critical points, a main point to address is the interchanging of expectation and Clarke subdifferential. Unfortunately, Clarke subdifferential theory is incompatible with this calculus rule, and sampling Clarke subgradients approximates an expected subgradient which may not be a subgradient of the expectation. Furthermore, oracles used in practical implementations may not be equal to Clarke subgradients. For instance, automatic differentiation, or backpropagation, used in deep learning \cite{lecun2015deeplearning}, may not be equal to a Clarke subgradient \cite{bolte2020mathematical}. For this reason, a recent model of generalized gradient called conservative gradient has been introduced in \cite{bolte2019conservative}. These objects are closed graph set-valued maps that ensure a descent mechanism along curves and that extend the calculus rules to nonsmooth functions. Usual differentiation rules apply to these gradient objects: for instance, composing conservative Jacobians gives a conservative Jacobian, and the sum of conservative gradients gives a conservative gradient for the sum. In particular, such a conservative gradient models the backpropagation oracle \cite{bolte2020mathematical} which is given by the chain rule on Clarke Jacobian, and also other implementations such as implicit differentiation \cite{bolte2021nonsmooth} which has many applications in machine learning like implicit neural networks \cite{deepeqmodels},  hyperparameter optimization \cite{pmlr-v119-bertrand20a} or latent variable models \cite{figurnovimplicit}.

Following this model, it is shown in \cite[Theorem 3.10]{bolte2022subgradient} that the expectation of a conservative gradient is a conservative gradient for the expectation. This fundamental property justifies subgradient sampling in practical implementations. Consequently,  it is proved in  \cite{bolte2022subgradient} that the stochastic subgradient method converges to generalized critical points, defined as the zero locus of a conservative gradient. Due to the construction of the conservative gradients, their values may differ from the Clarke subdifferential. Such a convergence result is then not satisfactory and one may converge to artificial critical points having absurd locations. Following ideas from \cite{bianchi2020,bolte2020mathematical}, it is shown in \cite{bolte2022subgradient} that for the stochastic subgradient method and a semialgebraic integrand, artifacts coming from the conservative calculus can be avoided with a randomized initialization. This result furthermore holds for a more general class of functions called \emph{definable in an o-minimal structure} \cite{vandendries1996,coste}. This is an exhaustive class of functions having stability properties under simple operations such as composition or summation, and that contains most functions met in practice: not only semialgebraic functions but also the exponential and the logarithm. As a consequence of this result, in most practical cases, accumulation points are Clarke critical for a randomized initialization and for all but a finite set of stepsizes \cite[Theorem 2.3.2]{bolte2022subgradient}.

We extend this analysis to a nonsmooth stochastic heavy ball method involving a stochastic conservative gradient oracle.  As a result, we propose a unifying setting that encompasses many modern implementations in machine learning. This general setting allows us to use different perspectives of the differential inclusion approach \cite{bianchi2021closed,benaim} in order to recover convergence results of some previous works \cite{Ruszczynski2020,bianchi2021closed}, and also to rectify certain limitations found in these prior studies:

- By using a conservative gradient oracle, we justify practical implementations of the method like gradient sampling with automatic differentiation output, or with the use of implicit differentiation. This technical aspect was not considered in \cite{bianchi2021closed} which assumes the access to a stochastic oracle centered around the Clarke subdifferential.

- It is often necessary to require that the critical values have an empty interior, a condition known as the Sard condition. This condition, crucial for obtaining convergence, is challenging to satisfy when dealing with objective functions defined as expectations. For instance, it arises under an abstract form in \cite{Ruszczynski2020}. We obtain convergence results even when the Sard condition is not met.  The first result we derive from \cite{benaim} is that accumulation points minimizing the Lyapunov function values are critical. Then, following the approach from \cite{bianchi2021closed}, we show the criticality of accumulation points that are consistently seen and called \emph{essential}.
When the sample distribution is absolutely continuous, we propose an assumption to ensure the Sard condition, which may be of interest. This assumption requires the density to be semialgebraic, or more generally \emph{globally subanalytic}: a large class of definable functions that contains analytic functions with polynomial growth, which may not be semialgebraic. This condition was recently introduced in the analysis of the stochastic subgradient method \cite{bolte2022subgradient}.

- Since it encompasses practical implementations as in \cite{bolte2022subgradient}, the analysis yields convergence results in a general sense and involves a notion of criticality depending on a conservative gradient that may differ from the usual Clarke subdifferential. A criticality notion of a similar kind also appeared in \cite{Ruszczynski2020} which uses the semismooth counterparts to conservative gradients, Norkin's generalized gradients \cite{norkin1980,norkin1986,ermol1998stochastic}. 
This particular aspect of the analysis highlights the possibility for accumulation points to be what we may call \emph{artificial critical points}, points that are critical with respect to the conservative gradient but not critical with respect to the Clarke subdifferential. We tackle this issue and demonstrate that randomizing initializations is sufficient to avoid nondifferentiability points and to achieve criticality in the usual sense, that is, with respect to the Clarke subdifferential. As a result, we complement the analysis conducted in \cite{Ruszczynski2020,bianchi2021closed} which did not address this particular question.

\paragraph{Organization of the paper.} The paper is organized as follows: in \Cref{section:mainresults}, we present our main convergence results stated in \Cref{th:convergenceconservative} and \Cref{th:convergenceclarke}. \Cref{section:materials} gathers necessary preliminaries for the differential inclusion method, conservative gradients, and for the understanding of definable sets.
In \Cref{section:nonsmoothanalysis} we prove our main convergence result \Cref{th:convergenceconservative} via the differential methods from \cite{benaim} and \cite{bianchi2021closed}. In \Cref{section:genericclarke}, we show the genericity of stochastic gradient sequences and then deduce \Cref{th:convergenceclarke}  on Clarke criticality when randomizing initializations.

\section{Main results}
\label{section:mainresults}
\paragraph{Notations.} The Clarke subdifferential of a locally Lipschitz function $F : \R^p \rightarrow \R$ is defined as $\partial^c F(w) := \conv \left\{\lim_{k \to \infty} \nabla F(w_k) \ : \  w_k \in \diff_F, w_k \underset{k \to \infty}{\to} w\right\}$ for all $w \in \R^p$, where $\operatorname{diff}_F$ is the differentiability set of $F$ which is of full measure in virtue of Rademacher's theorem. Given a function $z : \R \to \R^p$ we denote the derivative of $z$ by $\dot{z}$ whenever it is defined. We denote the Euclidean norm by $\|\cdot \|$. For a subset $A \subset \R^p$, $\conv A$ is its convex hull, $\dim A$ is the Hausdorff dimension of $A$, and we denote $\|A\| := \sup \{\|y\| \ : \ y \in A \}$. Given $x \in \R^p$, we call \emph{neighborhood} of $x$ any open subset of $\R^p$ (with respect to the Euclidean norm) containing $x$.

\paragraph{Framework.} We consider a stochastic minimization problem
\begin{equation}
    \label{eq:minimization}
    \min_{w \in \R^p} F(w) = \E_{\xi \sim P}[f(w, \xi)],
\end{equation}
where we assume $F$ to be well defined on $\R^p$ and bounded below by $F^* := \inf_{w \in \R^p} F(w) > - \infty$. We are interested in a stochastic heavy ball method \eqref{eq:HBoneequation}: 
\begin{equation}
    \label{eq:HBoneequation}
    w_{k+1} = w_k - \mu_k v(w_k, \xi_k) + \nu_k (w_k - w_{k-1}),
\end{equation}
where $(\xi_k)_{k \in \N}$ is an i.i.d sequence of random variables valued in a measurable space $(S, \mathcal{A})$. We denote by $P$ the law of $\xi_k$ for $k \in \N$. The function $v$ represents a stochastic first-order oracle to be defined in this section. As an example that falls in our analysis,   $v(\cdot, s)$ may be taken as a selection of the Clarke subdifferential of $f(\cdot, s)$ for almost all $s \in S$.

For the analysis to come, we shall also use the following form of the recursion \eqref{eq:HBoneequation}:

\begin{equation}
\label{eq:HBupdate}
\begin{aligned}
    w_{k+1}  & = w_{k} - \alpha_k y_{k}\\ 
    y_{k+1}  & = \beta_k v(w_{k+1}, \xi_{k+1}) + (1 - \beta_k) y_k.
\end{aligned}
\end{equation}
One can easily verify \eqref{eq:HBupdate} reduces to \eqref{eq:HBoneequation} with $\mu_k = \alpha_k \beta_{k-1}$ $\nu_k = \alpha_k(1 - \beta_{k-1})/\alpha_{k-1}$, and $y_0 = \frac{w_0 - w_1}{\alpha_0}$. The sequence $(w_{k})_{k \in \N}$ is adapted to the filtration generated by $(\xi_k)_{k \in \N}$, i.e., for all $k \in \N$, $w_k$ is measurable with respect to $\xi_k, \ldots, \xi_0$. 

\bigskip

We consider bounded trajectories of $(w_k)_{k \in \N}$ and make assumptions on the stepsizes $(\alpha_k)_{k \in \N}$ and $(\beta_k)_{k \in \N}$.
\begin{assumption}
\label{ass:stochasticalgorithm}
\begin{enumerate}
    \item[]
    \item $\sup_{k \in \N} \|w_k\| < \infty$ almost surely,
    \item For all $k \in \N$, $\alpha_k >0$, $\sum_{k =0}^{\infty} \alpha_k = \infty$ and  $\sum_{k=0}^\infty \alpha_k^2 < \infty$,
    \item There exists $r > 0$ such that $\alpha_k/\beta_k  \to r$, and $ \beta_k \in (0,1)$ for all $k \in \N$.
\end{enumerate}
\end{assumption}
\Cref{ass:stochasticalgorithm}.1 is common in stochastic algorithms in the nonconvex setting. Some solutions have been proposed in the literature to guarantee the boundedness of the sequence, although none of them is satisfactory. On this question, we refer the reader to a discussion provided in \cite{bolte2022subgradient}. \Cref{ass:stochasticalgorithm}.2 is a common assumption on the stepsizes in the stochastic approximation literature. As to \Cref{ass:stochasticalgorithm}.3, it is referred to as exponential memory in \cite{gadat2018stochastic}. Although we restrict our analysis to this case, other parametrizations of the memory coefficient may be considered like a polynomial one, see for instance \cite{gadat2018stochastic,pmlr-v115-orvieto20a}.

We assume access to a stochastic first-order oracle defined thanks to the notion of conservative gradient:
\begin{assumption}[Stochastic conservative gradient] \label{ass:stochasticconservativegradient}
\begin{enumerate}
    \item[]
    \item $f$ and $v$ are jointly measurable, and for almost all $s \in S$, $f(\cdot, s)$ and $v(\cdot, s)$ are semialgebraic.
    \item For almost all $s \in S$, $f(\cdot, s)$ is locally Lipschitz. $v$ is a selection in a convex-valued and measurable set-valued map $D : \R^p \times S \rightarrow \R^p$ which is such that for almost all $s \in S$, $D(\cdot, s)$ is a conservative gradient for $f(\cdot, s)$.
    \item There exists  a square integrable function $\kappa : S \to \R_+$ and $\psi : \R_+ \to \R_+$ nondecreasing and locally bounded such that for almost all $s \in S$, $\|D(w, s)\| \leq \kappa(s) \psi(\|w\|)$  for all $w \in \R^p$. 
\end{enumerate}
\end{assumption}

The notion of conservative gradient, introduced in \cite{bolte2019conservative}, will be properly explained in \Cref{subsection:setvalued_conservative}. Considering this kind of oracle allows us to cover practical implementations of the method \eqref{eq:HBoneequation}. While $D$ can be taken as the classical Clarke subdifferential $\partial^c_w f$ with respect to $w$, it can also model the backpropagation oracle, used in deep learning, as well as automatic differentiation based on implicit differentiation \cite{bai2019}.

The expected conservative gradient will be central to our analysis and from now on, we set
$$D_F := \E_{\xi \sim P}[D(\cdot, \xi)].$$ 

Let us remark that the expectation is taken on a set-valued map, in which case the appropriate definition of the integral, due to Aumann (\Cref{def:aumann}), is used here.

Although a part of our results only relies on the assumptions above, we may obtain stronger convergence results by assuming the following on the distribution $P$:

\begin{assumption}[Semialgebraic distribution] \label{ass:sard} $S=\R^m$, $f$ and $D$ are jointly semialgebraic, and one of the following conditions holds:
\begin{itemize}
    \item[--] $P$ is finitely discrete.
    \item[--] $P$ has a semialgebraic density with respect to Lebesgue.
\end{itemize}
\end{assumption}
This assumption aims to satisfy the Sard condition for the expectation $F$, see \Cref{subsection:definablesets}. On the Sard condition, we also refer the reader to \cite{shikhman}.

\begin{remark}[On definable functions]
    To lighten the presentation, we only consider semialgebraic functions in \Cref{ass:stochasticconservativegradient} and \Cref{ass:sard}. This setting might seem quite restrictive as it doesn't capture some functions used in machine learning such as the exponential or the logarithm. To address this concern, functions that are definable in an o-minimal structure can be considered instead of mere semialgebraic ones. We discuss this possibility in \Cref{subsection:definablesets}.

\end{remark}

Without \Cref{ass:sard}, it is possible to obtain convergence results in the sense of \emph{essential} accumulation points. This notion of accumulation points was introduced in \cite{bianchi2021closed,Bolte2020LongTD}. Intuitively, they are accumulation points that are consistently seen in the long run:

\begin{definition}[Essential accumulation points] Let $(w^*, y^*)$ be an accumulation point of $(w_k,y_k)_{k \in \N}$. $(w^*, y^*)$ is called essential if almost surely,
          $$\limsup_{k \to \infty}\frac{\sum_{i=0}^{k} \alpha_i \mathbbm{1}_{(w_k,y_k) \in U}}{\sum_{i=0}^k \alpha_i}>0 \text{ for all neighborhood $U$ of  $(w^*, y^*)$}.$$
\end{definition}

\bigskip

We are now ready to present our main convergence results:

\begin{theorem}[Conservative criticality]
\label{th:convergenceconservative}
 Let $(w_k,y_k)_{k \in \N}$ be generated by \eqref{eq:HBupdate}. Under \Cref{ass:stochasticalgorithm} and \Cref{ass:stochasticconservativegradient}, the following hold almost surely in the sequence $(\xi_k)_{k \in \N}$:
\begin{enumerate}
    \item (Essential criticality). Every essential accumulation point $(w^*, y^*)$ of the sequence $(w_k,y_k)_{k \in \N}$ satisfies $0 \in D_F(w^*)$ and $y^* = 0$.
    \item (Minimal criticality). Every accumulation point $(w^*, y^*)$ of $(w_k,y_k)_{k \in \N}$ such that \linebreak $\underset{k \to \infty}{\liminf} \: F(w_k) + \frac{r}{2} \|y_k\|^2 = F(w^*) + \frac{r}{2} \|y^*\|^2$ satisfies $0 \in D_F(w^*)$ and $y^* = 0$.
    \item (Objective function convergence and criticality). If \Cref{ass:sard} also holds, then every accumulation point $(w^*, y^*)$ of $(w_k,y_k)_{k \in \N}$ satisfies  $0 \in D_F(w^*)$ and $y^* = 0$, and $F(w_k)$ converges as $k \to \infty$.
\end{enumerate}
\end{theorem}

As for the stochastic subgradient method \cite{bolte2022subgradient}, the generalized criticality $0 \in D_F(w^*)$ is inherent to subgradient sampling. Nonetheless, under proper randomization of the initialization, we can retrieve criticality with respect to  the Clarke subdifferential:

\begin{theorem}[Clarke criticality under randomized initializations] \label{th:convergenceclarke} Let $(w_k,y_k)_{k \in \N}$ be generated by \eqref{eq:HBupdate}. Under \Cref{ass:stochasticalgorithm} and \Cref{ass:stochasticconservativegradient}, there exists a set $W \subset \R^p \times \R^p$ of full measure  such that if  $(w_1, w_0) \in W$,  then almost surely, \Cref{th:convergenceconservative}.1-2 hold with $\partial^c F$ inplace of $D_F$.  If \Cref{ass:sard} also holds, then  \Cref{th:convergenceconservative}.3 holds with $\partial^c F$ in place of $D_F$.

If $P$ has a density with respect to Lebesgue, and $f$ and $v$ are jointly semialgebraic, then $W^c$ is a countable union of manifolds of dimension at most $2p-1$.
    
\end{theorem}

\section{Preliminary materials}
\label{section:materials}

In this section, we present some necessary definitions for the convergence analysis to come in \Cref{section:nonsmoothanalysis}. A first subsection will be dedicated to the differential inclusion method that we use in order to analyze the nonsmooth stochastic heavy ball method \eqref{eq:HBupdate}. The definitions of this subsection come from the general framework \cite{benaim} and also intervene in the closed-measure perspective \cite{bianchi2021closed}.  A second subsection will be dedicated to sets that are definable in an o-minimal structure, a generalization of semialgebraic sets that allows to capture the simple geometry of functions used in practice.

\subsection{Differential inclusion method and conservative gradients}
\label{subsection:setvalued_conservative}

Given a set-valued map $H : \R^p \rightrightarrows \R^p$, we are interested in recursions of the form
\begin{equation}
    \label{eq:discreterecursion}
    z_{k+1} \in z_k + \alpha_k (H^{\delta_k}(z_k) + \epsilon_k),  \hspace{1cm} \text{for all } k \in \N,
\end{equation}
where $(\epsilon_k)_{k \in \N}$ is a martingale difference sequence, $(\alpha_k)_{k \in \N}$ are vanishing stepsizes and $(\delta_k)_{k \in \N}$ is a positive and random sequence such that $\delta_k \to 0$ as $k \to 0$ almost surely. For $\delta  > 0$, $H^\delta$ is the fattened map:
\begin{align*}
    H^\delta(z) := & \left\{  y \in \R^p :  \exists  (z', y') \in \R^p \times \R^p, \right. \\ & \left.
y' \in H(z'), \|z' -  z\| \leq \delta, \|y'- y\| \leq \delta \right\}.
\end{align*}

The (stochastic) subgradient method \cite{Bolte2020LongTD,bolte2022subgradient} and the stochastic heavy ball method \eqref{eq:HBupdate} fall in this setting.

\paragraph{Limiting dynamical system and Lyapunov function.}

The asymptotic behavior of the recursion \eqref{eq:discreterecursion} is captured by a differential inclusion, a generalization of ordinary differential equations to a set-valued vector field:
\begin{equation}
\label{eq:DIdeff}
    \dot{z} \in H(z).
\end{equation}
This type of dynamical system naturally arises when dealing with nonsmooth optimization algorithms due to the set-valued aspect of generalized gradients for nonsmooth functions, such as the Clarke subdifferential. In order to ensure the existence of solutions for the differential inclusion \eqref{eq:DIdeff}, the \emph{upper semicontinuity} of $H$ is often required. This condition is also equivalent to $H$ having a \emph{closed graph} and being \emph{locally bounded}.

\begin{definition} Let $H : \R^p \rightrightarrows \R^p$ be a set-valued map.
\begin{itemize}
    \item (Upper semicontinuity) $H$ is called upper semicontinuous if for all $\epsilon > 0$, for all $x \in \R^p$, there exists $\delta > 0$ such that for all $y \in B(x, \delta)$,  $H(y) \subset H(x) + B(0, \epsilon)$.
    \item (Closed graph) $H$ is said to have a closed graph if its graph, which is the set $\{(x,y) \in \R^p \times \R^p \ : \ y \in H(x) \}$, is closed in $\R^p \times \R^p$.
    \item (Local boundedness) $H$ is locally bounded if for all compact set $K \subset \R^p$, there exists $M > 0$ such that for all $x \in K$, $\|y\| \leq M$ for all $y \in H(x)$.
\end{itemize}

\end{definition}

Assuming furthermore that $H$ is nonempty, compact, and convex-valued, one has the local existence of solutions for the differential inclusion \eqref{eq:DIdeff}. Then, by preventing explosion in finite time, a local solution can be extended to a global solution, defined on $\R_+$. For a complete theoretical development on differential inclusions, see for instance \cite{aubin,Filippov1988DifferentialEW}.

A solution to \eqref{eq:DIdeff} is defined as an absolutely continuous function $z : \R_+ \to \R^p$ such that $\dot{z}(t) \in H(z(t))$ holds for almost all $t \geq 0$. We may then define the set-valued flow $\Phi$ for $z_0 \in \R^p$ and $t \geq 0$:
\begin{equation*}
    \Phi_t(z_0) := \{z(t) : z \text{ is a solution to \eqref{eq:DIdeff} and } z(0) = z_0 \}.
\end{equation*}
We also use the notation $\Phi(U)$ to denote the set of solutions starting from a subset $U \subset \R^p$.

A \emph{Lyapunov function} characterizes the stability of the limiting differential inclusion.

\begin{definition}[Lyapunov function] \label{def:lyapunov} A continuous function  $E : \R^p \to \R$ is a Lyapunov function for a set $\Lambda \subset \R^p$ and for the differential inclusion \eqref{eq:DIdeff} if
\begin{align*}
    &\forall x \in \R^p \setminus \Lambda \;, \forall t > 0, \forall y \in \Phi_t(x), \: E(y) <  E(x), \\ &\forall x \in \Lambda,\; \forall t \geq 0, \forall y \in \Phi_t(x), \: E(y) \leq E(x).
\end{align*}
\end{definition}

\paragraph{Interpolated process and perturbed solution.} In order to relate the discrete recursion \eqref{eq:discreterecursion} to its continuous time version \eqref{eq:DIdeff}, one may consider the interpolated process $\Bar{z} : \R_+ \to \R^p$ associated to the sequence $(z_k)_{k \in \N}$. We recall its construction:

Let $\tau_0 = 0$ and for all $k \in \N^*$, $\tau_k := \sum_{i=0}^{k-1} \alpha_i$. Then define the continuous function $\Bar{z}$ so that for each $k \in \N$, $\Bar{z}(\tau_k) = x_k$ and $\Bar{z}$ is affine on $[\tau_k, \tau_{k+1}]$.

Note that the limit set of the interpolated process is equal to the accumulation points of $(z_k)_{k \in \N}$. Under some conditions on the noise sequence and the stepsizes, the interpolated process can be characterized as a \emph{perturbed solution}:

\begin{definition}[Perturbed solution] \label{def:perturbedsol} An absolutely continuous function $x$ is called perturbed solution of (\ref{eq:DIdeff}) if there exists a locally integrable function $U : \R_+ \to \R^p$ satisfying
\begin{equation*}
    \underset{t \to \infty}{\lim} \underset{v \in [0,T]}{\sup} \left\|\int_{t}^{t+v} U(s) \di s \right\| =0,
\end{equation*}
and a positive function $\delta : \R_+ \to \R_+$ satisfying $\delta(t) \to 0$ as $t \to \infty$, such that for almost all $t > 0$, $\dot{x}(t) \in H^{\delta(t)}(x(t)) + U(t)$.
\end{definition}

\paragraph{Conservative gradients.}

\Cref{ass:stochasticconservativegradient} uses a generalized notion of gradients called \emph{conservative gradients} \cite{bolte2019conservative}.  They are locally bounded set-valued maps having closed graphs that satisfy a chain rule along the curves:

\begin{definition}[Conservative gradient and path differentiability] \label{def:conservativegradient} A set-valued map $D_F : \R^p \rightrightarrows \R^p$ is a conservative gradient for $F$ if it has closed graph, nonempty valued, locally bounded and for all absolutely continuous curve $\gamma : \R_+ \rightarrow \R^p$, the chain rule 
    \begin{equation*}
        \frac{\di (F \circ \gamma)(t)}{\di t} = \langle v, \dot{\gamma}(t)  \rangle, \text{ for all } v \in D_F(\gamma(t)) 
    \end{equation*}
holds for almost all $t \in \R_+$. In this case, $F$ is called \emph{path differentiable}.
\end{definition}

Conservative Jacobians can also be defined similarly, see \cite[Definition 4]{bolte2019conservative}. The chain rule property is well-suited to the differential inclusion approach as it ensures a descent mechanism and allows for the existence of a Lyapunov function for the limiting dynamical system. As a consequence of this property, the objective function is a Lyapunov function in the case of the subgradient method. As to the stochastic heavy ball method \eqref{eq:HBupdate}, a Lyapunov function may be the total energy and will be exposed thereafter in \Cref{section:nonsmoothanalysis}.

Conservative derivatives enable the extension of calculus rules to nonsmooth functions. For instance, given two path differentiable functions $f$ and $g$ with conservative gradients $D_f$ and $D_g$ respectively, then  $D_f + D_g$ is a conservative gradient for $f + g$. A chain rule for compositions also holds and justifies backpropagation, which comes from applying the chain rule formula to Clarke Jacobians. Precisely, given a path differentiable function $G : \R^p \rightarrow \R$ and the output $v(w)$ of backpropagation on $G(w)$, we may find a conservative  gradient $D_G$ such that $v(w) \in D_G(w)$ for all $w$, see \cite{bolte2020mathematical,bolte2019conservative}. Other differentiation rules such as nonsmooth implicit differentiation can also be represented by conservative derivatives \cite{bolte2021nonsmooth}.

A property of conservative gradients that is central in this work, is the interchanging of expectation and conservative gradient \cite[Theorem 3.10]{bolte2022subgradient} which justifies sampling with nonsmooth oracles used in practice, such as automatic differentiation.

Let us first recall the notion of expectation for set-valued maps:

\begin{definition}[Aumann's integral] \label{def:aumann} Let $D : \R^p \times S \rightrightarrows \R^p$ be measurable. Then the expectation of $D$ with respect to $P$ is defined for all $w \in \R^p$ as
\begin{align*}
    \E_{\xi \sim P}[D(w, \xi)] := \{ \int_S h(w,s) \di P (s)  \ :  \   & h(w, \cdot) \text{ is integrable }, \\ & h(w,s) \in D(w,s)  \text{ for all } s \in S \}.
\end{align*}
\end{definition}

Conservative differentiation of an expectation is stated as follows:

\begin{theorem} \cite[Theorem 3.10]{bolte2022subgradient} \label{th:integralconservative} Let $D : \R^p \times S \rightrightarrows \R^p$ satisfy \cref{ass:stochasticconservativegradient}. Then $D_F = \E_{\xi \sim P}[D(\cdot, \xi)]$ is a conservative gradient for $F$.
\end{theorem}

The distinctive feature of the conservative gradients model is that it doesn't derive from a formula involving the function, in contrast to the Clarke subdifferential.  Instead, it arises from the property of the set-valued map to account for the variations of the function. As a result, we may have many candidates to be conservative gradients given the same function. Despite this non-usual aspect, this variational model is a faithful representation of the practice. Indeed, in modern implementations, first-order oracles are built upon a representation of the function to differentiate. For instance, backpropagation is applied to a compositional expression,  subgradient sampling uses the function's representation as an expectation, and implicit differentiation uses the property of the function to be the solution path of an equation. 

\paragraph{Artifacts and artificial critical points.} For a path differentiable function $F$, conservative gradients are equal to $\nabla F$ Lebesgue almost everywhere \cite[Theorem 1]{bolte2019conservative}. In particular, this implies that the Clarke subdifferential is the minimal convex-valued conservative gradient, and any conservative gradient is equal to the Clarke subdifferential Lebesgue almost everywhere. Given a conservative gradient $D_F$, we may call  \emph{artifacts} the set where $D_F \neq \partial^c F$. When analyzing first-order methods with a conservative gradient $D_F$, a generalized notion of criticality comes out and may include \emph{artificial critical points}, which are points where $0 \in D_F(a)$ but $0 \notin \partial^c F(a)$. These points can have absurd locations. Indeed, for an arbitrary point $a$, one may find a conservative gradient $D$ that satisfies $0 \in D(a)$. Such a conservative gradient can be constructed by changing the value of $\partial^c F(a)$ to $\partial^c F(a) \cup \overline{B(0,1)}$.

This observation has motivated some previous works \cite{bolte2020mathematical,bianchi2020,bolte2022subgradient} demonstrating that for the stochastic subgradient method, artifacts are often avoided and criticality with respect to the Clarke subdifferential is often achieved in place of the general criticality notion $0 \in D_F(w^*)$. In \Cref{section:genericclarke}, we extend this line of research to the case of the stochastic heavy ball method \eqref{eq:HBupdate}, leading to \Cref{th:convergenceclarke}.

It's worth noting that this concern was already highlighted in \cite{ermol1998stochastic} where the authors used the semismooth generalized gradients to study the stochastic subgradient method. However, their solution to avoid the artifacts involved injecting a uniform noise, which may not align with practical scenarios.

\subsection{Semialgebraic, definable sets and functions}
\label{subsection:definablesets}

Considering semialgebraic sets, or more generally \emph{o-minimal structures}, aims to exclude pathological cases that rarely occur in practical situations. Sets that belong to an o-minimal structure $\mathcal{O}$, which is a collection of subsets from all spaces $\R^n$, for $n \in \N$, are called \emph{definable} in $\mathcal{O}$. By extension, functions are called definable in $\mathcal{O}$ if their graphs belong to $\mathcal{O}$.

This subsection serves as a concise presentation of definable sets and their properties. Main o-minimal structures of interest in machine learning applications will be presented, but their comprehensive definitions are omitted here. For more details, we refer the reader to \cite[Subsection 4.1]{bolte2022subgradient} or  \cite[Appendix A.2]{bolte2021nonsmooth} for a more operational view in machine learning. For more complete references, we refer the reader to \cite{coste,vandendries1996}.

Semialgebraic sets are the simplest example of such a structure. Still, it encompasses many functions used in machine learning: ReLU, convolution layers,  max pooling,  the square loss, or the $\ell_1$ regularization. O-minimal structures are generalizations of semialgebraic sets based on their stability properties by simple operations such as  finite unions, intersections or by projection.

Two o-minimal structures are relevant to consider in machine learning applications: \emph{globally subanalytic} sets, also denoted $\Ran$, and $\Ranexp$. The reason to consider them is to include some functions that are recurrently used in machine learning, but that are not semialgebraic, such as the exponential and the logarithm.

Globally subanalytic sets, $\Ran$, is an o-minimal structure containing analytic functions having polynomial growth, and also the restriction of analytic functions to semialgebraic compact sets. In particular, the exponential function defined on $\R$ is not included in this class of function, but its restriction to any compact interval is. The o-minimal structure $\Ranexp$ contains both the globally subanalytic sets and the exponential function.

By using these o-minimal structures, some of the assumptions from \Cref{section:mainresults} may be relaxed:

\begin{itemize}
    \item In \Cref{ass:stochasticconservativegradient}, $f$ and $v$ can be assumed to be definable in the same arbitrary o-minimal structure, for instance, $\Ranexp$.
    \item In \Cref{ass:sard}, if $P$ is finite, one can assume more generally $f$ and $D$ to be definable in the same arbitrary o-minimal structure. But if  $P$ is absolutely continuous, one can only assume $f$, $D$, and the density of $P$ to be at most globally subanalytic.
\end{itemize}

In particular, \Cref{th:convergenceconservative,th:convergenceclarke} hold under these relaxed conditions. Remark that in the case of an absolutely continuous distribution, we can't assume $f$, $D$, and the density of $P$ to be definable in an arbitrary o-minimal structure. Indeed, \Cref{ass:sard} aims to obtain the Sard condition by using a result on stability under integration  \cite{Cluckers_2011} that mostly applies to globally subanalytic functions.

\begin{remark}[On the use of the term ``definable"] With an abuse of terminology, and when there is no confusion, we will often call a set or function ``definable",  without specifying the underlying o-minimal structure which is assumed to be the same for all objects. In the assumptions, the reader can also consider ``definable" to be ``semialgebraic" for simplicity. 
\end{remark}

\paragraph{Stratification of definable sets and consequences.} An important property of definable sets  is the stratification into $C^r$ manifolds:

\begin{theorem}[Stratification of a definable set \cite{vandendries1996}] \label{th:stratification} Let $A \subset \R^p$ be  definable. Then there exists a partition $(M_k)_{k = 1, \ldots, q}$ of $A$ such that for $k = 1, \ldots, q$, $M_k$ is  a $C^r$ manifold  and  for all couple $i,j$, $\overline{M_i} \cap M_j \neq \emptyset \Rightarrow M_j \subset \overline{M_i}$. Such a partition is called a $C^r$-stratification.
\end{theorem}

This can be seen as a partition where each manifold's boundary contains other manifolds of lower dimension. As an illustration, the stratification of the closed cube is given by the open cube, the faces, the edges, and the vertices. 

The stratification property underlines the simple geometric structure of definable sets and it has important consequences in nonsmooth optimization. For instance, if $f : \R^p \rightarrow \R$ is definable, we may apply it to the graph of $f$ in order to have a set $L \subset \R^p$ which is a dense finite union of open manifolds on which $f$ is $C^r$. A refined form of stratification called \emph{Whitney stratification}, see e.g. \cite{coste,vandendries1996}, has two important consequences in nonsmooth nonconvex optimization:

\begin{itemize}
    \item (Path differentiability). Locally Lipschitz and definable functions are path differentiable  \cite[Proposition 2]{bolte2019conservative}.
    \item (Sard condition). A locally Lipschitz and definable function $F$, accompanied with a definable conservative gradient $D_F$, satisfies the Sard condition: the set of $D_F$-critical values, $F(\{w \in \R^p \ : \ 0 \in D_F(w) \})$, has empty interior \cite[Theorem 5]{bolte2019conservative}.
\end{itemize}

Some remarkable consequences regarding the Lebesgue measure of definable sets are worth mentioning for our analysis. For instance, a definable open and dense set is a finite union of open manifolds and has full Lebesgue measure. Definable zero measure sets are necessarily a finite union of low-dimensional manifolds and in dimension one, they are finite. We also have the following result on definable dense products which we use for our avoidance results in \Cref{section:genericclarke}:

\begin{lemma}[Dense product {\cite[Lemma 4.9]{bolte2022subgradient}}] \label{prop:denseproduct} Let $(l, m) \in \N^* \times \N^*$ and $L \subset \R^l \times \R^m$ be definable. Then $L$ is dense in $\R^l \times \R^m$ if and only if there exists a definable dense set  $A \subset \R^l$ such that for all $w \in A$, the set $\{z \in \R^m \ : \ (w,z) \in L\}$ is dense in $\R^m$.
\end{lemma}

\section{Nonsmooth analysis of stochastic heavy ball}
\label{section:nonsmoothanalysis}

In this part, we show the convergence of the method via differential inclusion approaches \cite{bianchi2021closed,benaim}. The analysis relies on the existence of a Lyapunov function for a limiting differential inclusion. We recall that $D_F = \E_{\xi \sim P}[D(\cdot, \xi)]$ and we consider the differential inclusion
\begin{equation}
    \label{eq:HBdifferentialinclusion}
    \begin{aligned}
        \dot{w} & = -ry\\
        \dot{y} & \in  D_F(w) - y,\\
    \end{aligned}
\end{equation}
and the total energy $E(w,y) := F(w) + \frac{r}{2}\|y\|^2$. $E$ is bounded below by $F^* = \inf_{w \in \R^p} F(w)$. These elements are common to the analysis of the heavy ball method and are recurrent in the literature, see e.g. \cite{gadat2018stochastic} for the smooth setting, \cite{Ruszczynski2020} that takes into account constraints, and \cite{bianchi2021closed} with the Clarke subdifferential oracle. 

We replace the usual gradient notions with the conservative gradient $D_F$ in order to include practical considerations such as the use of sampling with automatic differentiation. We demonstrate that $E$ serves as a Lyapunov function according to \Cref{def:lyapunov}, facilitating direct connections with various differential approaches \cite{bianchi2021closed,benaim} for a unified and comprehensive study.

In order to prevent explosion in finite time, linear growth can be assumed on $H$, as in \cite{benaim,bianchi2021closed}. This condition may seem a bit restrictive in deep learning applications since neural networks functions often have polynomial growth, hence the use of the more general \Cref{ass:stochasticconservativegradient}.3. We will see in the next result that the boundedness of the Lyapunov function can be used to ensure the existence of solutions.

\begin{proposition}[Existence of solution and Lyapunov function] \label{prop:lyapunovfunction} Under   \Cref{ass:stochasticconservativegradient}, the differential inclusion \eqref{eq:HBdifferentialinclusion} admits solutions and $E$ is a Lyapunov function for \eqref{eq:HBdifferentialinclusion} and $\Lambda : =\{x \in \R^p \ : \ 0 \in D_F(x) \} \times \left\{0\right\}$.
\end{proposition}

\begin{proof}
\emph{Existence of solution.} Since $D_F$ is a conservative gradient by \Cref{th:integralconservative}, it is upper semicontinuous, hence the right-hand side in  \eqref{eq:HBdifferentialinclusion} is upper semicontinuous and \eqref{eq:HBdifferentialinclusion}  admits a local absolutely continuous solution $(w,y)$ on $[0,T]$ for some $T > 0$. In order to show it can be extended to a solution on $\R_+$, it is sufficient to show an explosion in finite time of the lengths of the curves $\int_{0}^T \|\dot{w}(t)\| \di t$ and  $\int_{0}^T \|\dot{y}(t)\| \di t$ cannot happen, see for instance  \cite[Theorem 2]{Filippov1988DifferentialEW}. By the definition of a conservative gradient, we have for almost all $t \in [0,T]$,
\begin{equation*}
    \frac{\di (E \circ (w,y))}{\di t} (t) = \langle D_F(w(t)), \dot{w}(t) \rangle + r \langle y(t), \dot{y}(t)\rangle.
\end{equation*}
Integrating from $0$ to $T$ gives
\begin{align}
    \label{eq:lyapunov}
    E(w(T), y(T)) - E(w(0), y(0)) & = \int_{0}^T \frac{\di (E \circ (w,y))}{\di t} (t) \di t \nonumber \\ & =   \int_0^T \left(\langle D_F(w(t)), \dot{w}(t) \rangle + r \langle y(t), \dot{y}(t)\rangle \right) \di t \nonumber \\
    & = \int_{0}^T \langle \dot{y}(t) + y(t), - r y(t) \rangle + r \langle y(t), \dot{y}(t) \rangle \di t \nonumber \\
    & = - \int_{0}^T r \|y(t)\|^2 \di t.
\end{align}
Hence for the component $w$, we have
\begin{align}
    \label{eq:curvelength_w}
    \int_{0}^T \|\dot{w}(t)\| \di t = \int_{0}^T r\|y(t)\| \di t & \leq  \int_{0}^T r(1 + \|y(t)\|^2) \di t \nonumber \\ & = rT + E(w(0), y(0)) - E(w(T), y(T)) \nonumber \\ & \leq rT + E(w(0), y(0)) - F^*,
\end{align}
which holds for any horizon $T > 0$. As to the component $y$, we have
\begin{align*}
    \int_0^T \|\dot{y}(t)\| \di t & \leq \int_0^T (\|D_F(w(t))\| + \|y(t)\|) \di t \\ 
    & =  \int_0^T (\|D_F(w(t))\| + \frac{1}{r}\| \dot{w}(t) \|) \di t \\ 
    & \leq \int_0^T \|D_F(w(t))\| \di t + T + \frac{1}{r}(E(w(0), y(0)) - F^*).
\end{align*}
Let $\kappa$ and $\psi$ be given by \Cref{ass:stochasticconservativegradient}.3, then $$\| D_F(w(t)) \| \leq \E_{\xi \sim}[\kappa(\xi)] \psi(\|w(t)\|) \leq \E_{\xi \sim}[\kappa(\xi)] \sup_{t \in [0,T]} \psi(\|w(t)\|)$$ 
for all $t \in [0,T]$. Finally,
\begin{equation}
    \label{eq:curve_length_y}
    \int_0^T \|\dot{y}(t)\| \di t \leq T\E_{\xi \sim P}[\kappa(\xi)] \sup_{s \in [0,T]} \psi(\|w(s)\|) + T + \frac{1}{r}(E(w(0), y(0)) - F^*).
\end{equation}
Since $\psi$ is locally bounded and by \eqref{eq:curvelength_w}, $\|w(s)\| \leq \|w_0\| + rs + E(w(0), y(0)) - F^*$ for all $s \in [0, T]$, \eqref{eq:curve_length_y} holds for any horizon $T >0$. Finally, the local solution $(w,y)$ can be extended to a global solution on $\R_+$.

\emph{Lyapunov function.} We now verify that $E$ is a Lyapunov function in the sense of \Cref{def:lyapunov}. Let $(w,y)$ be an absolutely continuous solution of \eqref{eq:HBdifferentialinclusion} with $(w(0), y(0)) \in\Lambda$. In this case, the equation \eqref{eq:lyapunov}, holds for $T > 0$, hence 
 $E(w(T), y(T)) - E(w(0), y(0)) \leq 0$.

Now suppose $(w(0), y(0))\notin\Lambda$. If $y(0) \neq 0$, we have $E(w(t), y(t)) < E(w(0), y(0))$ from  \eqref{eq:lyapunov}, by continuity of $y$. If $0 \notin D_F(w(0))$, suppose toward a contradiction that there exists $t > 0$ such that $E(w(t), y(t))=E(w(0), y(0))$. It means by \eqref{eq:lyapunov} that $y(s) = 0$ for all $s \in [0, t]$, hence $\dot{y}(s) = 0$  and then $0 \in D_F(w(s))$ for all $s \in [0, t]$. Since $D_F$ has a closed graph, we would have $0 \in D(w(0))$ which is a contradiction.

Finally, for $(w(0), y(0))\notin\Lambda$, one has for almost all $t > 0$ , $ E(w(t), y(t)) < E(w(0), y(0))$.
\end{proof}

\bigskip

For the sequence $(w_k, y_k)_{k \in \N}$ defined by \eqref{eq:HBupdate}, we denote for all $k \in \N^*$, $V(w_{k}) = \E_{\xi \sim P}[v(w_{k}, \xi)]$ and $u_{k} = v(w_{k}, \xi_{k}) - V(w_{k})$. With these notations, the second equation in \eqref{eq:HBupdate}  writes

\begin{equation}
\label{eq:update_y_forproof}
    y_{k+1} = (1 - \beta_k) y_k + \beta_k V(w_{k+1}) + \beta_k u_{k+1}.
\end{equation}

\begin{remark}[Clarke subgradient sequences] \label{remark:replacebyclarke} Whenever almost surely $V(w_k) \in \partial^c F(w_k)$ for all $k \in \N$, all the results of this part also hold replacing $D_F$ by $\partial^c F$ in the differential inclusion \eqref{eq:HBdifferentialinclusion}.
    
\end{remark}

We show that the sequence $(w_k, y_k)_{k \in \N}$ is a \emph{perturbed solution} (see \cite[Definition II]{benaim}) of the differential inclusion \eqref{eq:HBdifferentialinclusion}. We first prove preliminary lemmas.

\begin{lemma}[Noise extinction] \label{lem:noiseextinction} Under \Cref{ass:stochasticalgorithm}, \Cref{ass:stochasticconservativegradient}, 
\begin{enumerate}
    \item $\sup_{k \in \N} \E[\|u_{k+1}\|^2 | \xi_k, \ldots, \xi_0] < \infty$ almost surely,
    \item $\sum_{i = 0}^\infty \beta_i u_{i+1}$ converges almost surely.
\end{enumerate}
\end{lemma}
\begin{proof} Item 1 is a consequence of \Cref{ass:stochasticconservativegradient}.3. Item 2 is an application of square integrable martingale convergence theorem, see for instance the proof of item 2 of \cite[Lemma 3.14]{bolte2022subgradient}.
\end{proof}

The following lemma appears as an assumption in \cite{bianchi2021closed}. Similarly to \cite[Lemma 1]{Ruszczynski2020}, we show it is a consequence of the boundedness of $(w_k)_{k \in \N}$.
\begin{lemma}[Velocity boundedness] \label{lem:velocityboundedness} Let $(w_k, y_k)_{k \in \N}$ be defined by \eqref{eq:HBupdate}. Under \Cref{ass:stochasticalgorithm} and \Cref{ass:stochasticconservativegradient}, $(y_k)_{k \in \N}$ is bounded almost surely.
\end{lemma}

\begin{proof}
     For all $k \in \N$, we define the quantity $\Tilde{y}_k := y_k + \sum_{i = k}^\infty \beta_i u_{i+1}$. Adding $\sum_{i = k+1}^\infty \beta_i u_{i+1}$ on both sides in \eqref{eq:update_y_forproof} gives 
    \begin{equation*}
        \Tilde{y}_{k+1} = (1 - \beta_k) \Tilde{y}_k + \beta_k( V(w_{k+1}) +  \Tilde{y}_k - y_k),
    \end{equation*}
    hence we have  $\|\Tilde{y}_{k+1}\| \leq \max  \{\|\Tilde{y}_k\|,  \| V(w_{k+1})\| + \| \Tilde{y}_k - y_k\|\}$ for all $k$. By direct recurrence we have $\|\Tilde{y}_{k+1}\| \leq \max \{\|\Tilde{y}_0\|, \| V(w_{i+1})\| + \| \Tilde{y}_i - y_i\|,  i = 0, \ldots , k\}$. Furthermore, we have almost surely $$\sup_{k \in \N} \| V(w_{k+1})\| <  \E_{\xi \sim P}[\kappa(\xi)]\sup_{k \in \N} \psi (\|w_{k+1}\|) < \infty$$ by \Cref{ass:stochasticalgorithm}.1 and \Cref{ass:stochasticconservativegradient}.3. By \Cref{lem:noiseextinction}.2, the quantity $\| \Tilde{y}_k - y_k\| = \|\sum^{\infty}_{i=k} \beta_i u_{i+1}\|$ goes to $0$ as $k \to \infty$. Finally, we have 
    $$\limsup_{k \to \infty} \|y_k\| \leq \max \{ \|\Tilde{y}_0\|, \sup_{k \in \N} \left( \| V(w_{k+1})\| + \| \Tilde{y}_k - y_k\| \right) \}< \infty.$$ 
\end{proof}

We are now ready to prove \Cref{th:convergenceconservative}.

\begin{proof}[ of \Cref{th:convergenceconservative}]
    
Let $(w_k, y_k)_{k \in \N}$ be defined by \eqref{eq:HBupdate}. In order to apply the setting from \Cref{subsection:setvalued_conservative} and the frameworks \cite{benaim,bianchi2021closed}, we first verify that the sequence satisfies a recursion of the form:

\begin{align}
    \label{eq:desired_update}
    w_{k+1} &= w_k - \alpha_k y_{k} \nonumber \\
    y_{k+1} &\in y_k + \frac{\alpha_k}{r}  ( D_F^{\delta_{k}}(w_{k}) -  y_k + u_{k+1}),
\end{align}
where almost surely, $\delta_k \to 0$ as $k \to \infty$.

For clarity, we may assume $r \beta_k = \alpha_k$.  We have for all $k \in \N$, $D_F(w_{k+1}) = D_F(w_k - \alpha_k y_k) \subset D_F^{\alpha_k \|y_k\|}(w_k)$.  By \Cref{ass:stochasticalgorithm}.2 and \Cref{lem:velocityboundedness}, $\alpha_k \|y_k\|$ goes to $0$ as $k \to \infty$, hence \eqref{eq:update_y_forproof} writes as the desired recursion \eqref{eq:desired_update} with $\delta_k := \alpha_k \|y_k\|$.

\bigskip

We  are now ready to prove points 1, 2, and 3 of \Cref{th:convergenceconservative}:

We first prove \Cref{th:convergenceconservative}.1. By \Cref{prop:lyapunovfunction}, $E : (w, y) \mapsto F(w) + \frac{r}{2} \|y\|^2$ is a Lyapunov function for  \eqref{eq:HBdifferentialinclusion} and the set $\Lambda = \{x \in \R^p \ : \ 0 \in D_F(x)\} \times \{0\}$.  All the conditions are satisfied to apply  \cite[Corollary 4.9]{bianchi2021closed}, in particular \cite[Assumption 4.1]{bianchi2021closed} is satisfied since $D_F$ is a conservative gradient, \cite[Assumption 4.2]{bianchi2021closed} is satisfied by \Cref{lem:noiseextinction} and \cite[Assumption 4.3]{bianchi2021closed} is satisfied by \Cref{ass:stochasticalgorithm}. Finally, every essential accumulation point $(w^*, y^*)$ of $(w_k, y_k)_{k \in \N}$ satisfies $0 \in D_F(w^*)$ and $y^* = 0$.

We then prove \Cref{th:convergenceconservative}.2. By \Cref{lem:noiseextinction}, $\sum_{i = 0}^{k} \beta_k u_{k+1}$ converges as $k \to \infty$ and we can apply \cite[Proposition 1.3]{benaim} to say the interpolated process associated to $(w_k, y_k)_{k \in \N}$ is almost surely a perturbed solution  of \eqref{eq:HBdifferentialinclusion}, see \Cref{def:perturbedsol}. We then follow the beginning of the proof of \cite[Proposition 3.27]{benaim}. The accumulation points of $(w_k, y_k)_{k \in \N}$, which we denote $L$, is a compact set. $L$ is furthermore invariant by the flow of \eqref{eq:HBdifferentialinclusion}, by \cite[Theorem 3.6]{benaim} and \cite[Lemma 3.5]{benaim}. Let $(w^*, y^*)$ be such that $E(w^*, y^*) = \min E(L)$. In particular,  $E(w^*, y^*) = \underset{k \to \infty}{\liminf}{\: F(w_k) + \frac{r}{2} \|y_k\|^2}$. Then if $(w,y)$ is a solution of \eqref{eq:HBdifferentialinclusion} starting from $(w^*, y^*)$, one has for all $t > 0$, $(w(t), y(t)) \in L$ and also, $E(w(t), y(t)) \geq E(w^*, y^*)$ by definition of $(w^*, y^*)$. Since $E$ is a Lyapunov function (\Cref{prop:lyapunovfunction}) for the set $\Lambda$, then $(w^*, y^*)$ has to be in $\Lambda$.

    We now prove \Cref{th:convergenceconservative}.3. Under \Cref{ass:sard}, $F = \E_{\xi \sim P}[f(\cdot, \xi)]$ and the set-valued map $D_F = \E_{\xi \sim P}[D(\cdot, \xi)]$ are definable. Indeed, for the discrete case, this is a consequence of the stability of semialgebraic functions by finite sum. In the case where $P$ has a semialgebraic density with respect to Lebesgue, this is a consequence of an integration theorem of globally subanalytic functions and set-valued maps, see \cite[Theorem 1.3]{Cluckers_2011} and \cite[Theorem 4.8]{bolte2022subgradient}. We may apply definable Sard's theorem for conservative gradients \cite[Theorem 5]{bolte2019conservative} to obtain that the set of $D_F$-critical values, $F(\{w \in \R^p \ : \ 0 \in D_F(w) \})$, has empty interior.

    We can  finally combine \cite[Theorem 3.6]{benaim} and  \cite[Proposition 3.27]{benaim} to say that every accumulation point $(w^*, y^*)$  of $(w_k, y_k)_{k \in \N}$ satisfies $0 \in D_F(w^*)$ and $y^* = 0$, and $E(w_k,y_k) = F(w_k) + \frac{r}{2}\|y_k\|^2$ converges to $F(w^*) +\frac{r}{2}\|y^*\|^2$, hence $F(w_k)$ converges as $k \to \infty$. 
\end{proof}

\section{Avoidance of nonsmooth artifacts}
\label{section:genericclarke}

In this part, we show that randomizing the initialization is sufficient to avoid all nonsmooth artifacts: nondifferentiability points and artifacts created by the conservative calculus, discussed in \Cref{subsection:setvalued_conservative}. Our results rely on the stratification property of definable sets and functions highlighted in \Cref{subsection:definablesets}.

We consider the second order recursion \eqref{eq:HBoneequation}, with stepsizes $(\mu_k)_{k \in \N}, (\nu_k)_{k \in \N}$:
\begin{equation*}
    w_{k+1} = w_k - \mu_k v(w_k, \xi_k) +  \nu_k(w_k - w_{k-1}),
\end{equation*}
where we recall $\mu_k = \alpha_k \beta_{k-1}$ $\nu_k = \alpha_k(1 - \beta_{k-1})/\alpha_{k-1}$, and $y_0 = \frac{w_0 - w_1}{\alpha_0}$.

We can write it as follows:
\begin{equation*}
    \begin{pmatrix}
    w_{k+1} \\
    w_{k}
    \end{pmatrix} = \begin{pmatrix}
        w_k - \mu_k v(w_k, \xi_k) + \nu_k (w_k - w_{k-1}) \\
        w_{k}
    \end{pmatrix} = \Psi_{\mu_k, \nu_k, \xi_k}(w_k, w_{k-1}),
\end{equation*}
where for $s \in S$ and $(x, y) \in \R^p \times \R^p$, we let $\Psi_{\mu, \nu,s}(x,y) := (x - \mu v(x, s) + \nu (x -y), x)$. With an abuse of notation and for simplicity, we will write for a sequence $(s_k)_{k \in \N}$, for all $k \in \N$, $\Psi_{\mu_k,\nu_k, s_k} = \Psi_{s_k}$. In this case, for all $k \in \N^*$, $(w_{k+1}, w_k) = (\Psi_{s_k} \circ \Psi_{s_{k-1}} \circ \ldots \circ \Psi_{s_{1}})(w_1, w_0)$.

\bigskip 

In the proofs to come in this section, we will use the following notations:

\begin{itemize}
    \item We denote $\Sigma \subset S$ of full measure given by \Cref{ass:stochasticconservativegradient}.1-3, which is such that for $s \in \Sigma$, $f(\cdot, s)$ and $v(\cdot, s)$ are semialgebraic, and $\|D(w,s)\| \leq \kappa(s) \psi(\|w\|)$ for all $w\in \R^p$.
    \item For all $s \in \Sigma$, we denote $R_s \subset \R^p$, given by \Cref{ass:stochasticconservativegradient}.1-2, the semialgebraic and dense subset such that for all $w \in R_s$, $\nabla_w f(\cdot, s) = v(\cdot, s)$, and $f(\cdot, s)$ is $C^2$ on a neighborhood of $w$. Such a subset is given by the property of a conservative gradient to be gradient almost everywhere, see \cite[Theorem 1]{bolte2019conservative}, and the stratification property of semialgebraic sets (\Cref{th:stratification}).
\end{itemize}

We start by proving the following lemma:

\begin{lemma} \label{lem:definablediffeo} Under \Cref{ass:stochasticconservativegradient}, for $s \in \Sigma$ and $\nu \neq 0$, one has for all semialgebraic subset $Z \subset \R^{p} \times \R^p$,
\begin{equation*}
    \dim Z \leq 2p -1 \Rightarrow \dim \Psi_{\mu, \nu, s}^{-1}(Z) \leq 2p -1. 
\end{equation*}
\end{lemma}
\begin{proof}
    Let $Z \subset \R^p$ be semialgebraic, of dimension at most $2p - 1$. Toward a contradiction, suppose $ \dim \Psi_{\mu, \nu, s}^{-1}(Z) = 2p$.  Since $\Psi_{\mu, \nu, s}^{-1}(Z)$ is semialgebraic, then by stratification, there exists an open subset $U$ of $\R^p \times \R^p$ included in $\Psi_{\mu, \nu, s}^{-1}(Z)$. On $R_s \times \R^p$, $\Psi_{\mu, \nu, s}(x,y) = (x - \mu \nabla_w f (x,s) + \nu(x-y), x)$, and its Jacobian is well defined, given by
\begin{equation*}
    \Jac \Psi_{\mu, \nu, s} (x,y) = \begin{pmatrix}
                I_p - \mu \nabla^2_w f(x,s) + \nu I_p &  & - \nu I_p \\
                 \\
                I_p &  & 0_p
    \end{pmatrix}.
\end{equation*}

$ \Jac \Psi_{\mu, \nu, s} (x,y)$ is clearly invertible whenever $\nu \neq 0$. In particular, $\Psi_{\mu, \nu, s}$ is a local diffeomorphism on $R_s \times \R^p$. Since $R_s \times \R^p$ is dense and open, and $U$ is open, $U \cap (R_s \times \R^p)$ has nonempty interior. This implies $\Psi_{\mu, \nu, s}(U \cap (R_s \times \R^p))$ has nonempty interior in $\R^p \times \R^p$, in particular it has dimension $2p$,  but it is included in $Z$ by definition of $U$, which is a contradiction since $\dim Z < 2p$. 
    
\end{proof}

Then, we can deduce the following by induction:

\begin{proposition}[Avoidance of nonsmooth set] \label{prop:measuregenericitygradient}
     Under \Cref{ass:stochasticconservativegradient}, for all $k \in \N^*$, there exists a subset $W_k \subset \R^p \times \R^p$ of full Lebesgue measure such that if $(w_1, w_0) \in W_k$, then for $(s_i)_{i=1, \ldots, k}$ in a set of full measure with respect to $P^k$, it holds that for $i = 1, \ldots, k$, $\nabla_w f(w_i, s_i) = v(w_i,s_i)$, and $\E_{\xi \sim P}[v(w_i, \xi)] = \nabla F(w_i)$. 
     
     In particular, for $(w_1, w_0)$ in a set of full Lebesgue measure $W \subset \R^p \times \R^p$,  $\E_{\xi \sim P}[v(w_k, \xi)] = \nabla F(w_k)$  for all $k \in \N^*$, $P^{\otimes \N}$-almost surely in $(s_k)_{k \in \N}$.
\end{proposition}

\begin{proof} We first show  that for $k \in \N^*$, for $P^k$-almost all $(s_i)_{i=1, \ldots, k}$, there exists a set $Z_k \subset \R^p \times \R^p$ of full measure such that for all $(w_1, w_0) \in Z_k$, $\nabla_w f(w_i, s_i) = v(w_i,s_i)$  for $i=0, \ldots, k$. We then deduce the desired result by Fubini's theorem.

For each $k \in \N^*$, assume $s_k \in \Sigma$. Fix $k \in \N^*$. For $i = 1, \ldots, k$, we let $V_i = (\Psi_{s_i} \circ \Psi_{s_{i-1}} \circ \ldots \circ \Psi_{s_1} )^{-1}(R_{s_{i+1}} \times R_{s_{i}})$. By construction, if $(w_1, w_0) \in V_i$, then $(w_{i+1}, w_i) \in R_{s_{i+1}} \times R_{s_i}$. Consequently, by definition of $R_s$ for $s \in \Sigma$, if $(w_1, w_0) \in Z_k := \bigcap_{i=1}^k V_k$ then for all $i = 1, \ldots k$, $\nabla_w f(w_i, s_i) = v(w_i,s_i)$. 

It remains to verify for all $k \in \N^*$, $V_k$ is dense semialgebraic, or $V_k^c$ has dimension at most $2p - 1$. This is done by induction with \Cref{lem:definablediffeo}.

Fix $k \in \N^*$, $R_{s_{k+1}} $ and $R_{s_k}$ are semialgebraic and dense, hence $ \dim (R_{s_{k+1}} \times R_{s_k})^c \leq 2p - 1$ by stratification. Since  $s_k \in \Sigma$, then by \Cref{lem:definablediffeo} we have that $\Psi_{s_k}^{-1}((R_{s_{k+1}} \times R_{s_k})^c)$ has dimension at most $2p - 1$. Applying recursively \Cref{lem:definablediffeo} for the function $\Psi_{s_i}$  and the set $(\Psi_{s_{k}} \circ \ldots \circ \Psi_{s_{i+1}})^{-1}((R_{s_{k+1}} \times R_{s_k})^c)$ for $i = k - 1$ to $1$, proves that $ V_k^c = (\Psi_{s_k} \circ \Psi_{s_{k-1}} \circ \ldots \circ \Psi_{s_1} )^{-1}((R_{s_{k+1}} \times R_{s_{k}})^c)$ has dimension at most $2 p  - 1$. Finally, the intersection $Z_k = \bigcap_{i=1}^k V_k$ is dense semialgebraic.

For any $k \in \N^*$, we proved that for $P^k$-almost all $(s_i)_{i=1, \ldots, k}$ in $\Sigma^k$, there exists  $Z_k$ of full measure such that if $(w_1, w_0) \in Z_k$, then  $\nabla_w f(w_i, s_i) = v(w_i,s_i)$  for $i=1, \ldots, k$. By Fubini's theorem on the product $\Sigma^k \times (\R^p \times \R^p)$, this implies that there exists a set of full Lebesgue measure $W_k \subset \R^p \times \R^p$, such that if $(w_1, w_0) \in W_k$, then for $P^k$-almost all $(s_i)_{i=1, \ldots, k}$ in $\Sigma^k$, $\nabla_w f(w_i, s_i) = v(w_i,s_i)$  for $i=1, \ldots, k$. By \Cref{ass:stochasticconservativegradient}.3, we have $\|\nabla_w f(w_i, s_i)\| \leq \kappa(s_i) \psi(\|w_i\|)$ for almost all $s_i$. By dominated convergence theorem, we can interchange integral with respect to $s_i$  and gradient with respect to $w_i$, to write $\nabla F(w_i) = \E_{\xi \sim P}[\nabla_w f (w_i, \xi)] = \E_{\xi \sim P}[v(w_i, \xi)]$.

Finally, set $W := \bigcap_{k \in \N^*} W_k$. By definition of the subsets $W_k$, for $k \in \N^*$, if $(w_1, w_0) \in W$, then $\E_{\xi \sim P}[v(w_k, \xi)] = \nabla F(w_k)$ for all $k \in \N^*$.

\end{proof}

Under a further assumption on the distribution $P$, the subset $W$ in \Cref{prop:measuregenericitygradient}, not only has full Lebesgue measure but is also \emph{residual}, meaning that its complement is a countable union of low dimensional manifolds.

\begin{corollary} \label{cor:definableLebesguegenericity} Assume $S = \R^m$, $P$ has a density with respect to Lebesgue and $f$ and $v$ are jointly semialgebraic. Then \Cref{prop:measuregenericitygradient} holds with the additional properties: $W_k^c$ is a finite union of manifolds with dimension at most $2p-1$ and $W^c$ is a countable union of manifolds with dimension at most $2p-1$.
\end{corollary}

\begin{proof}
    For $k \in \N^*$, consider the set 
    \begin{align*}
        L_k := \{ (w_1, w_0, s_1, \ldots, s_k) \ : \ & \forall i \in \{1, \ldots, k\}, v(w_i, s_i) = \nabla_w f(w_i, s_i)\}.
    \end{align*}
    
   By stability properties of semialgebraic sets, $L_k$ is semialgebraic. By \Cref{prop:measuregenericitygradient}, for Lebesgue almost all $(w_1,w_0)$ and for almost   all $(s_1, \ldots, s_k)$, $ v(w_i, s_i) = \nabla_w f(w_i, s_i)$ for $i = 1, \ldots, k$. Applying \Cref{prop:denseproduct}, $L_k$ is dense. Also by \Cref{prop:denseproduct}, there exists a semialgebraic set $W_k$ open and dense such that if $(w_1, w_0) \in W_k$, then it holds that $\forall i \in \{1, \ldots, k\}, v(w_i, s_i) = \nabla_w f(w_i, s_i)$ for almost all $s_i$. Under \Cref{ass:stochasticconservativegradient}.3, we then may apply the dominated convergence theorem as in the proof of \Cref{prop:measuregenericitygradient} to obtain $\nabla F(w_i) = \E_{\xi \sim P}[v(w_i, \xi)]$.

   By stratification, $W_k^c$ is a finite union of manifolds with dimension at most $2p-1$, and the complementary of $W = \bigcap_{k \in \N} W_k $ is a countable union of manifolds with dimension at most $2p-1$. 
\end{proof}

We can now deduce \Cref{th:convergenceclarke}:

\begin{proof}[of \Cref{th:convergenceclarke}]
We are given $D$ such that $D(\cdot, s)$ is a conservative gradient for $f(\cdot, s)$. 
    
    Now let $W$ be given by \Cref{prop:measuregenericitygradient}. Assume $(w_1, w_0) \in W$, then by definition of $W$, we have $P^{\otimes \N}$-almost surely, for all $k \in \N$,
    \begin{equation}
        \begin{aligned}
           & w_{k+1} = w_k - \alpha_k y_k \\
           & y_{k+1} = \beta_k (\nabla F(w_{k+1}) + u_{k+1})  + (1 - \beta_k) y_k,\\
        \end{aligned}
    \end{equation}
with $u_{k+1} = v(w_{k+1}, \xi_{k+1}) - \E_{\xi \sim P}[v(w_{k+1}, \xi)]$. In particular, $(y_k)_{k \in \N}$  satisfies the relation \eqref{eq:update_y_forproof} with $V(w_{k+1}) \in \partial^c F(w_{k+1})$.

Finally, as mentioned in \Cref{remark:replacebyclarke} we may follow the proof of \Cref{th:convergenceconservative} from \Cref{section:nonsmoothanalysis}  with $\partial^c F$ in place of $D_F$ to obtain
\begin{enumerate}
    \item Every essential accumulation point $(w^*, y^*)$ satisfies $0 \in \partial^c F(w^*)$ and $y^* = 0$.
    \item Every accumulation point $(w^*, y^*)$ such that $\underset{k \to \infty}{\liminf} F(w_k) + \frac{r}{2} \|y_k\|^2 = F(w^*) + \frac{r}{2} \|y^*\|^2$ satisfies $0 \in \partial^c F(w^*)$ and $y^* = 0$.
    \item If \Cref{ass:sard} also holds, then 1 holds for all accumulation points, and $F(w_k)$ converges as $k \to \infty$.
\end{enumerate}
The last statement of the theorem is simply a consequence of \Cref{cor:definableLebesguegenericity}. 

\end{proof}

\section{Conclusion}
We analyzed a stochastic heavy ball method in the nonsmooth and nonconvex setting using the differential inclusion approach.  An interchange rule for conservative gradient was used in order to justify subgradient sampling and to derive first results on the convergence to critical points with respect to the expected conservative gradient. Thanks to the Sard condition obtained through semialgebraic assumptions on the sample distribution, we obtained the convergence of the objective function. This first part of our analysis suggested that subgradient sampling and other calculus rules could impact the convergence, possibly leading to artificial critical points. We showed that it doesn't happen in general as artifacts coming from the conservative calculus are almost surely avoided for randomized initializations, hence retrieving a standard convergence to Clarke critical points.

\section*{Acknowledgements}
The author would like to thank Jérôme Bolte and Edouard Pauwels for their precious feedback. This work has benefitted from the AI Interdisciplinary Institute ANITI. ANITI is funded by the French ``Investing for the Future – PIA3" program under the Grant agreement n°ANR-19-PI3A-0004.

\bibliographystyle{spmpsci}
\bibliography{references}
\end{document}